\documentclass[10pt, reqno]{amsart}

\usepackage{amsmath,amssymb,amsthm,amsfonts,verbatim}
\usepackage{microtype}
\usepackage[all,2cell]{xy}
\usepackage{mathtools}
\usepackage{graphicx}
\usepackage{hyperref}
\usepackage{mathrsfs}

\CompileMatrices

\usepackage[top=1.3in,bottom=1.9in,left=1.4in,right=1.4in]{geometry}

\theoremstyle{plain}
\newtheorem{theorem}{Theorem}[section]
\newtheorem{question}[theorem]{Question}
\newtheorem{problem}[theorem]{Problem}
\newtheorem{proposition}[theorem]{Proposition}
\newtheorem{lemma}[theorem]{Lemma}

\theoremstyle{definition}

\newtheorem{example}[theorem]{Example}

\newtheorem{remark}[theorem]{Remark}

\newcommand{\nc}{\newcommand}
\nc{\dmo}{\DeclareMathOperator}

\nc{\Q}{\mathbb{Q}}
\nc{\F}{\mathbb{F}}
\nc{\R}{\mathbb{R}}
\nc{\Z}{\mathbb{Z}}
\nc{\C}{\mathbb{C}}
\nc{\Ell}{\mathcal{L}}
\nc{\M}{\mathcal{M}}
\nc{\K}{\mathcal{K}}
\nc{\disk}{\mathbb{D}}
\nc{\hyp}{\mathbb{H}}

\nc{\CP}{\mathbb{CP}}
\nc{\cS}{\mathcal{S}}
\dmo{\Mod}{Mod}
\dmo{\Diff}{Diff}
\dmo{\Homeo}{Homeo}
\dmo{\dist}{dist}
\dmo\BDiff{BDiff}
\dmo\SO{SO}
\dmo\Hom{Hom}
\dmo\SL{SL}
\dmo\Sp{Sp}
\dmo\rank{rank}
\dmo\sig{sig}
\dmo\Out{Out}
\dmo\Aut{Aut}
\dmo\Inn{Inn}
\dmo\GL{GL}
\dmo\PSL{PSL}
\dmo\tr{tr}
\dmo\BHomeo{BHomeo}
\dmo\EHomeo{EHomeo}
\dmo\EDiff{EDiff}
\nc\Sig{\Sigma}
\dmo\Teich{Teich}
\dmo\Fix{Fix}
\nc{\pair}[1]{\langle #1 \rangle}
\nc{\abs}[1]{\left| #1 \right|}
\nc{\action}{\circlearrowright}
\nc{\norm}[1]{\left | \left | #1 \right | \right |}
\nc{\abcd}[4]{\left(\begin{array}{cc} #1 & #2 \\ #3 & #4 \end{array}\right)}
\nc{\into}\hookrightarrow
\dmo{\Isom}{Isom}
\nc{\normal}{\vartriangleleft}
\dmo{\Vol}{Vol}
\dmo{\im}{Im}
\dmo{\Push}{Push}
\dmo{\Conf}{Conf}
\dmo{\PConf}{PConf}
\dmo{\id}{id}

\renewcommand{\epsilon}{\varepsilon}
\nc{\coloneq}{\mathrel{\mathop:}\mkern-1.2mu=}
\nc{\margin}[1]{\marginpar{\scriptsize #1}}
\nc{\para}[1]{\medskip\noindent\textbf{#1.}}

\newenvironment{packed_enum}{
\begin{enumerate}
  \setlength{\itemsep}{0pt}
  \setlength{\parskip}{0pt}
  \setlength{\parsep}{0pt}
}{\end{enumerate}}

\title{Surface bundles over surfaces with arbitrarily many fiberings}

\author{Nick Salter}
\email{nks@math.uchicago.edu}
\date{\today}

\address{Department of Mathematics\\ University of Chicago\\ 5734 S. University Ave., Chicago, IL 60637}

\begin{document}
\maketitle
\begin{abstract}
In this paper we give the first example of a surface bundle over a surface with at least three fiberings. In fact, for each $n \ge 3$ we construct $4$-manifolds $E$ admitting at least $n$ distinct fiberings $p_i: E \to \Sigma_{g_i}$ as a surface bundle over a surface with base and fiber both closed surfaces of negative Euler characteristic. We give examples of surface bundles admitting multiple fiberings for which the monodromy representation has image in the Torelli group, showing the necessity of all of the assumptions made in the main theorem of our recent paper \cite{salter}. Our examples show that the number of surface bundle structures that can be realized on a $4$-manifold $E$ with Euler characteristic $d$ grows exponentially with $d$.

\end{abstract}

\section{Introduction}
Let $M^3$ be a $3$-manifold fibering over $S^1$ with fiber $\Sigma_g\ (g \ge 2)$. If $b_1(M) \ge 2$, Thurston\footnote{While the theory of the Thurston norm gives the most complete picture of the ways in which a $3$-manifold fibers over $S^1$, earlier examples of this phenomenon were found by J. Tollefson \cite{tollefson} and D. Neumann \cite{neumann}.} showed that there are in fact infinitely many ways to express $M$ as a surface bundle over $S^1$, with finitely many fibers of each genus $h \ge 2$ \cite{thurstonnorm}. In contrast, F.E.A. Johnson showed that every surface bundle over a surface $\Sigma_g \to E^4 \to \Sigma_h$ with $g,h \ge 2$ has at most finitely many fiberings (see \cite{FEA2}, \cite{hillman}, \cite{rivinfiber} or Proposition \ref{proposition:bound} for various accounts). It is possible to deduce from Johnson's work that there is a universal upper bound on the number of fiberings that any surface bundle over a surface $E^4$ can have, as a function of the Euler characteristic $\chi(E)$. Specifically, Proposition \ref{proposition:bound} shows that if $E^4$ satisfies $\chi(E) = 4d$, then $E$ has at most $\sigma_0(d) (d+1)^{2d+6}$ fiberings as a surface bundle over a surface, where $\sigma_0(d)$ denotes the number of positive divisors of $d$.

The simplest example of a surface bundle over a surface with multiple fiberings
\footnote{The most straightforward notion of ``distinction'' for fiberings is that of fiberwise diffeomorphism. In this paper, we will also have occasion to consider a strictly stronger notion known as ``$\pi_1$-fiberwise diffeomorphism''. See Section \ref{section:construction} for the precise definition of $\pi_1$-fiberwise diffeomorphism, and see Proposition \ref{proposition:fiberings}, as well as Remark \ref{remark:pi1}, for a discussion of why we adopt this convention.} is that of a product $\Sigma_g \times \Sigma_h$, which has the two projections onto the factors $\Sigma_g$ and $\Sigma_h$. Prior to the results of this paper, there was essentially one general method for constructing nontrivial examples of surface bundles over surfaces with multiple fiberings, and they all yielded bundles with only two known fiberings (although it is in theory possible that these examples could admit three or more, cf Question \ref{question:known}). Such examples were first constructed by Atiyah and Kodaira (see \cite{atiyah}, \cite{kodaira}, as well as the account in \cite{moritabook}), and proceeded by taking a fiberwise branched covering of particular ``diagonally embedded'' submanifolds of products of surfaces.

It is worth remarking that if one is willing to relax the requirement that both the base and fiber surface have negative Euler characteristic, then it is possible to construct examples of $4$-manifolds $E$ admitting infinitely many fibrations over the torus $T^2$. If $M^3$ is a $3$-manifold admitting infinitely many fibrations over $S^1$, then $E = M^3 \times S^1$ has the required properties. However, Johnson's result indicates that when $g,h \ge 2$, the situation is necessarily much more rigid and correspondingly richer. The mechanism by which $E = M^3 \times S^1$ admits infinitely many fiberings is completely understood via the theory of the Thurston norm. In contrast, in the case $g,h \ge 2$, entirely new phenomena will necessarily occur. 

This paucity of examples, combined with the interesting features of the known constructions, led to the author's interest in surface bundles over surfaces with multiple fiberings. In \cite{salter}, the author established the following theorem which shows a certain rigidity among a particular class of surface bundles over surfaces. Let $\Mod_g$ denote the mapping class group of the closed surface $\Sigma_g$, and let $\mathcal{I}_g$ denote the {\em Torelli group}, i.e. the subgroup of $\Mod_g$ that acts trivially on $H_1(\Sigma_g, \Z)$. The {\em Johnson kernel} $\mathcal K_g$ is defined to be the subgroup of $\mathcal I_g$ generated by the set of Dehn twists about separating simple closed curves. Recall that the {\em monodromy} of a surface bundle $\Sigma_g \to E \to B$ is the homomorphism $\rho: \pi_1 B \to \Mod_g$ recording the mapping class of the diffeomorphism obtained by transporting a fiber around a loop in the base. 

\begin{theorem}[Uniqueness of fiberings: \cite{salter}, Theorem 1.2]\label{theorem:kg}
Let $\pi: E \to B$ be a surface bundle over a surface with monodromy in the Johnson kernel $\mathcal{K}_g$. If $E$ admits two distinct structures as a surface bundle over a surface then $E$ is diffeomorphic to $B \times B'$, the product of the base spaces. In other words, any nontrivial surface bundle over a surface with monodromy in $\K_g$ fibers as a surface bundle in a unique way. 
\end{theorem}

This result would seem to reinforce the impression that surface bundles over surfaces with multiple fiberings are extremely rare, and that examples with three or more fiberings should be even more exotic. However, the constructions of this paper show that there is in fact a great deal of flexibility in constructing surface bundles over surfaces with many fiberings. The following is a summary of the constructions given in Section \ref{section:construction}. 

\begin{theorem}[Existence of multiple fiberings]\label{theorem:main}
\ \ 

\begin{packed_enum}
\item\label{item:exist} For each $n \ge 3$ and each $g_1 \ge 2$ there exists a $4$-manifold $E$,  integers $g_2, \dots, g_n$ (which can be chosen so that $g_1, \dots g_n$ are pairwise distinct), and maps $p_i: E \to \Sigma_{g_i} (i = 1, \dots, n)$ realizing $E$ as the total space of a surface bundle over a surface in at least $n$ ways, distinct up to $\pi_1$-fiberwise diffeomorphism. If $g_i \ne g_j$, the fibers of $p_i$ and $p_j$ have distinct genera; consequently $p_i$ and $p_j$ are inequivalent up to fiberwise diffeomorphism whenever $g_i \ne g_j$. 
\item\label{item:torelli} There exist constructions as in (\ref{item:exist}) for which at least one of the monodromy representations $\rho_i: \pi_1 \Sigma_{g_i} \to \Mod_{h_i}$ has image contained in the Torelli group $\mathcal{I}_{h_i} \le \Mod_{h_i}$. 
\item \label{item:exp} There exists a sequence of surface bundles over surfaces $E_n$ for which $\chi(E_n) = 24n - 8$ and such that $E_n$ admits $2^n$ fiberings as a surface bundle over a surface, distinct up to $\pi_1$-fiberwise diffeomorphism.
\end{packed_enum}
\end{theorem}

The bound of Proposition \ref{proposition:bound} makes it sensible to define the following function:
\[
N(d) := \max\left\{\begin{array}{c|c}
n 	&	\mbox{there exists $E^4, \chi(E) \le 4d$, $E$ admits $n$ surface bundle structures}\\
 	& \mbox{distinct up to $\pi_1$-fiberwise diffeomorphism.}\end{array}
\right\}
\]
Phrased in these terms, (\ref{item:exp}) of Theorem \ref{theorem:main}, in combination with the upper bound of Proposition \ref{proposition:bound} implies that
\[
 2^{(d+2)/6} \le N(d) \le \sigma_0(d) (d+1)^{2d+6},
\]
where $\sigma_0(d)$ denotes the number of positive divisors of $d$. This should be compared to the previous lower bound $N(d) \ge 2$.

An additional corollary of Theorem \ref{theorem:main} is that it demonstrates the optimality of Theorem \ref{theorem:kg}. The {\em Johnson filtration} is a natural filtration $\mathcal{I}_g(k)$ on $\Mod_g$ recording how mapping classes act on nilpotent quotients of $\pi_1 \Sigma_g$. The first three terms in the filtration are given by $\mathcal{I}_g(1) = \Mod_g$, and $\mathcal{I}_g(2) = \mathcal{I}_g$, and $\mathcal{I}_g(3) = \mathcal{K}_g$. It follows from Theorem \ref{theorem:main} (\ref{item:torelli}) that Theorem \ref{theorem:kg} is optimal with respect to the Johnson filtration, in that there exist surface bundles over surfaces with multiple fiberings with monodromy contained in $\mathcal{I}_g$ and $\Mod_g$. 

 \para{Acknowledgements} I would like to thank Dan Margalit for extending an invitation to the 2014 Georgia Topology Conference where this project was begun, for bringing the Korkmaz construction to our attention and for other helpful discussions and comments. I would also like to thank Inanc Baykur, Jonathan Hillman, Andy Putman, and Bena Tshishiku for helpful conversations and suggestions. I would like to thank the referees for their detailed suggestions and corrections. Lastly I would like to thank Benson Farb for his continued support and guidance, as well as for extensive comments and corrections on this paper.

\section{The examples}
\para{The basic construction}\label{section:construction}
To illustrate our general method we start by describing a construction of a surface bundle over a surface $E$ admitting four fiberings $p_1, p_2, p_3, p_4: E \to \Sigma_g$. The monodromy of this bundle was first considered by Korkmaz\footnote{Unpublished; communicated to the author by D. Margalit.}, as an example of an embedding of a surface group inside the Torelli group. Related constructions were also used by Baykur and Margalit to construct Lefschetz fibrations that are not fiber-sums of holomorphic ones in \cite{margalit}. For what follows it will be necessary to give a direct topological construction of the total space. 

The method of construction is to perform a ``section sum'' of two surface bundles over surfaces (see \cite{margalit2} for a discussion of the section sum operation, including an equivalent description on the level of the monodromy representation). Let $\Sigma_{g_1} \to M_1 \to \Sigma_{h}$ and $\Sigma_{g_2} \to M_2 \to \Sigma_h$ be two surface bundles over a base space $\Sigma_h$, and for $i = 1,2$ let $\sigma_i: \Sigma_h \to M_i$ be sections of $M_1, M_2$. If the Euler numbers of $\sigma_1,\sigma_2$ are equal up to sign, then it is possible to perform a fiberwise connect-sum of $M_1, M_2$ along tubular neighborhoods of $\im(\sigma_i)$ (possibly after reversing orientation), giving rise to a surface bundle $\Sigma_{g_1 + g_2} \to M \to \Sigma_h$. In what follows, we will give a more detailed description of this construction and explain how it can be used to produce surface bundles over surfaces with many fiberings. 

\begin{remark}
We have chosen to present an example here where all of the fiberings have the same genus. In fact, the four fiberings presented here are equivalent up to fiberwise diffeomorphism, but {\em not} up to $\pi_1$-fiberwise diffeomorphism. We stress here that this is {\em not} an essential feature of the general method of construction described in the paper, but merely the simplest example which requires the least amount of cumbersome notation. See Remark \ref{remark:pi1} for more on why $\pi_1$-fiberwise diffeomorphism is an important notion of equivalence for our purposes, and see Theorem \ref{theorem:fibervariant} for the most general method of construction, which can produce $4$-manifolds that fiber as surface bundles in arbitrarily many ways with surfaces of distinct genera. It is worth noting that if $E^4$ fibers as a $\Sigma_g$-bundle and a $\Sigma_h$-bundle, for $g \ne h$, then clearly these two fiberings are distinct, up to bundle isomorphism, fiberwise diffeomorphism, or $\pi_1$-fiberwise diffeomorphism, since the fibers are not even homeomorphic!
\end{remark}

For $g \ge 2$, consider the product bundle $E_1 = \Sigma_g \times \Sigma_g$ with projection maps $p_V, p_H: E_1 \to \Sigma_g$ onto the first (resp. second) factor. Let $N$ be an open tubular neighborhood of the diagonal $\Delta$. The manifold $E$ is then constructed as the double
\[
E = (E_1 \setminus N) \cup_{\partial \overline{N}} (E_1 \setminus N),
\]
where the boundary components $\partial N$ are identified via the identity map. We let $E^+, E^-$ denote the ``upper'' and ``lower'' copies of $E_1 \setminus N$ contained in $E$. See Figure \ref{figure:e2sketch}.

\begin{center}
\begin{figure}
\includegraphics[scale = 0.75]{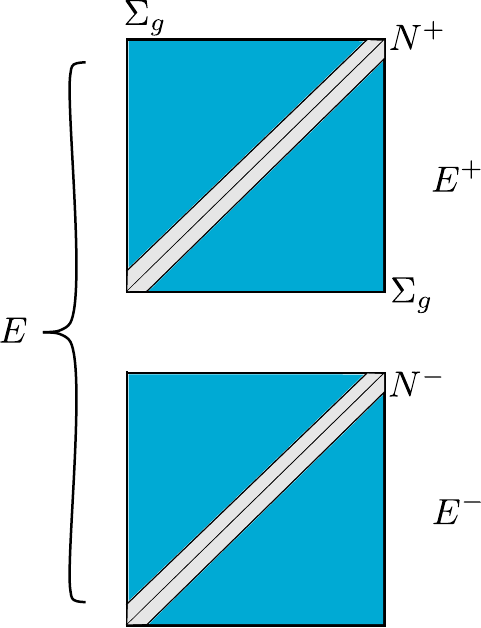}
\caption{A cartoon rendering of $E$, depicted as shaded. The boundaries are identified.}
\label{figure:e2sketch}
\end{figure}
\end{center}

$E$ is equipped with four fiberings $p_1, p_2, p_3, p_4: E \to \Sigma_g$. These correspond to the four combinations of horizontal and vertical fiberings on $E^+$ and $E^-$. For each $p_i$, we will exhibit collar neighborhoods of $\partial E^\pm$ relative to which the given $p_i$ will be smooth. 

To describe these collar neighborhoods, we endow $\Sigma_g$ with the structure of a Riemann surface. Via uniformization, this gives rise to a Riemannian metric, inducing a path metric $d$ on $\Sigma_g$. Relative to $d$, there is a neighborhood $N$ of the diagonal $\Delta$, for suitably small $\epsilon$, given via
\[
N = \{(z,w) \in \Sigma_g \times \Sigma_g \mid d(z,w) < \epsilon\}.
\]
The boundary $\partial \overline{N}$ is parameterized via the Riemannian exponential map $\exp_z$ at each $z \in \Sigma_g$ (for convenience we locally parameterize the circle of radius $\epsilon$ about $0 \in T_z \Sigma_g$ using the complex exponential):
\begin{align*}
\partial \overline N &= \{(z,w) \mid \abs{z -w} = \epsilon\}\\
 & = \{(z, \exp_z(\epsilon e^{i \theta})) \mid \theta \in [0, 2 \pi)\}\\
 & = \{(\exp_z(-\epsilon e^{i \theta}), z) \mid \theta \in [0, 2 \pi)\}.
\end{align*}

$p_1$ is defined using the vertical projection $p_V$ on each component. A suitable collar neighborhood (on either component) is given locally (for $t \in [1, 2)$) by
\[
\theta_V(z,\epsilon, t) = (z, \exp_z(t\epsilon e^{i \theta})).
\]

Similarly $p_2$ is defined using the horizontal projection $p_H$ on each component. A suitable collar neighborhood of either boundary component is now given locally (again for $t \in [1,2)$) by
\[
\theta_H(z,\epsilon, t) = (\exp_z(-t\epsilon e^{i \theta}), z).
\]

The remaining projections $p_3, p_4$ are defined using $p_V$ on one component and $p_H$ on the other. To realize these as smooth maps it will be necessary to modify the choice of boundary identification made in the construction of $E$. Consider the isotopy $h_t: \partial \overline{N} \times [0,1] \to \partial \overline{N}$ given locally by

\[
h_t(z, \exp_z(\epsilon e^{i \theta})) = (\exp_z(-t \epsilon e^{i \theta}), \exp_z((1-t) \epsilon e^{i\theta}))).
\]
More intrinsically, $h_t$ acts by rigidly translating the pair $(z,w)$ a distance $t \epsilon$ along the geodesic ray from $w$ to $z$; from this point of view it is clear that $h_t$ is a diffeomorphism, and so $h$ is indeed an isotopy. 

As $h_0 = \id$, there is a diffeomorphism
\[
f: E \to (E_1 \setminus N) \cup_{h_1} (E_1 \setminus N).
\]

$p_3$ is defined on $(E_1 \setminus N) \cup_{h_1} (E_1 \setminus N)$ using $p_V$ on the first component and $p_H$ on the second. Note that
\[
p_V(z, \exp_z(\epsilon e^{i \theta})) = (p_H \circ h_1) (z, \exp_z(\epsilon e^{i \theta})) = z,
\] 
so $p_3$ is well-defined. Moreover, $p_3$ is smooth relative to the collar neighborhoods $\theta_V$ on the first component and $\theta_H$ on the second. 

Completely analogously, $p_4$ is defined on $(E_1 \setminus N) \cup_{h_1} (E_1 \setminus N)$ using $p_H$ on the first component and $p_V$ on the second. See Figures \ref{figure:newfibering} and \ref{figure:newfibering2} for some depictions of the fibering $p_4$. 

It is clear that each $p_i$ is a proper surjective submersion; consequently by Ehresmann's theorem each $p_i$ realizes $E$ as the total space of a fiber bundle. In each case the base space is $\Sigma_g$, while the fiber is $\Sigma_g \# \Sigma_g \cong \Sigma_{2g}$. \\

We next recall the notion of {\em $\pi_1$-fiberwise diffeomorphism} from \cite{salter}. We say that two fiberings $p_1: E \to B_1$, $p_2: E \to B_2$ of a surface bundle are {\em $\pi_1$-fiberwise diffeomorphic} if 
\begin{packed_enum}
\item The bundles $p_1: E \to B_1$ and $p_2: E \to B_2$ are fiberwise diffeomorphic. That is, there exists a commutative diagram
\[
\xymatrix{
E \ar[r]^\phi \ar[d]_{p_1}		& E \ar[d]^{p_2}\\
B_1 \ar[r]_\alpha			& B_2
}
\]
with $\phi, \alpha$ diffeomorphisms.

\item The induced map $\phi_*$ preserves $\pi_1 F_1$, i.e. $\phi_*(\pi_1 F_1) = \pi_1 F_1$ (here, as always, $F_i$ denotes a fiber of $p_i$). 
\end{packed_enum}

In \cite{salter}, we gave the following criterion for two bundle structures to be distinct up to $\pi_1$-fiberwise diffeomorphisms (Proposition 2.1 of that paper):
\begin{proposition}\label{proposition:fiberings}
Suppose $E$ is the total space of a surface bundle over a surface in two ways: $p_1: E \to B_1$ and $p_2: E \to B_2$. Let $F_1, F_2$ denote fibers of $p_1, p_2$ respectively. Then the following are equivalent:
\begin{enumerate}
\item The fiberings $p_1, p_2$ are $\pi_1$-fiberwise diffeomorphic.
\item The fiber subgroups $\pi_1 F_1, \pi_1 F_2 \le \pi_1 E$ are equal.
\end{enumerate}
If $ \deg (p_1 \times p_2) \ne 0$ then the bundle structures $p_1$ and $p_2$ are distinct. 
\end{proposition}

With this characterization in mind, we will establish the following theorem.
\begin{theorem}\label{theorem:distinct}
The fiberings $p_i: E \to \Sigma_g$ for $i = 1,2,3,4$ constructed above are pairwise distinct up to $\pi_1$-fiberwise diffeomorphisms. 
\end{theorem}

\begin{proof}
To show that the projections $p_i$ as defined are pairwise distinct, we will appeal to condition (2) of Proposition \ref{proposition:fiberings}. For each $i$, the long exact sequence in homotopy of a fibration reduces to a short exact sequence
\[
\xymatrix{
1 \ar[r]	& \pi_1 F_i \ar[r]	& \pi_1 E \ar[r]^{p_{i,*}}	& \pi_1 \Sigma_g \ar[r]	& 1.
}
\]
To show that $\pi_1 F_i$ and $\pi_1 F_j$ are distinct for distinct $i,j$, it therefore suffices to produce an element $x \in \pi_1 F_i$ such that $p_{j,*}(x) \ne 1$ in $\pi_1 \Sigma_g$. 
Let $i$ and $j$ be distinct. Without loss of generality, suppose that $p_i$ is defined via $p_V$ on $E^+$, while $p_j$ is defined on $E^+$ via $p_H$. Let $F_i$ and $F_j$ denote generic fibers of $p_i, p_j$ respectively. Both of $F_i \cap E^+$ and $F_j \cap E^+$ are homeomorphic to $\Sigma_g^1$, the surface of genus $g$ and one boundary component.

 Let $\gamma \subset \Sigma_g^1$ be a non-peripheral loop representing a nontrivial element of $\pi_1 \Sigma_g^1$, and identify $\gamma$ with a loop in $F_i$. Then $[\gamma] \in \pi_1 F_i$ by construction (and is nontrivial), while $p_{j}(\gamma) = p_{H}(\gamma) = \gamma$. Here $\gamma$ is viewed as a loop in $\Sigma_g$ under the natural inclusion of $\Sigma_g^1$. As $\gamma$ was chosen to be non-peripheral and essential in $\Sigma_g^1$, it remains homotopically nontrivial in $\Sigma_g$. It follows that $\pi_1 F_i$ and $\pi_1 F_j$ are distinct for all distinct $i,j \in \{1,2,3,4\}$. Per Proposition \ref{proposition:fiberings}, $p_i$ and $p_j$ are not $\pi_1$-fiberwise diffeomorphic as claimed.
\end{proof}

\begin{center}
\begin{figure}
\includegraphics[scale = 0.8]{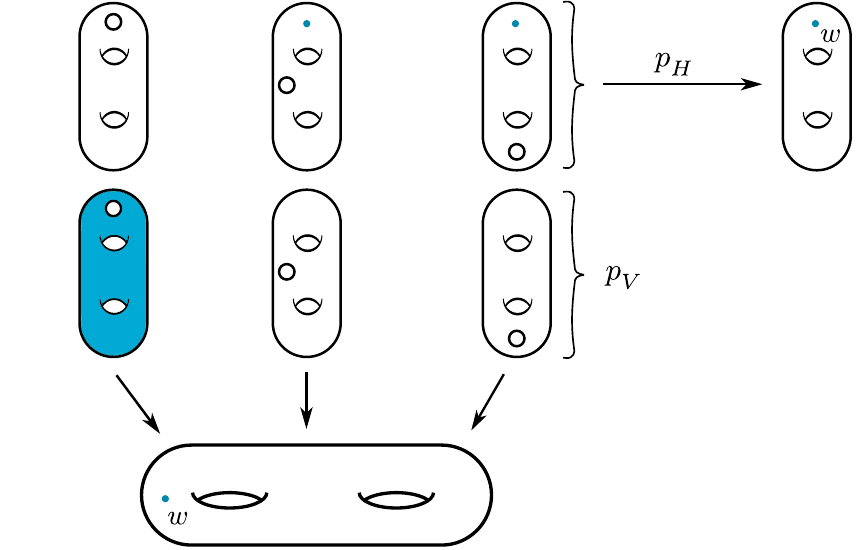}
\caption{The fibering $p_4: E \to \Sigma_g$. The fiber over $w \in \Sigma_g$ is shaded. On the upper portion of the bundle it intersects each of the $p_V$-fibers in a single point.}
\label{figure:newfibering}
\end{figure}
\end{center}

\begin{center}
\begin{figure}
\includegraphics[scale = 0.8]{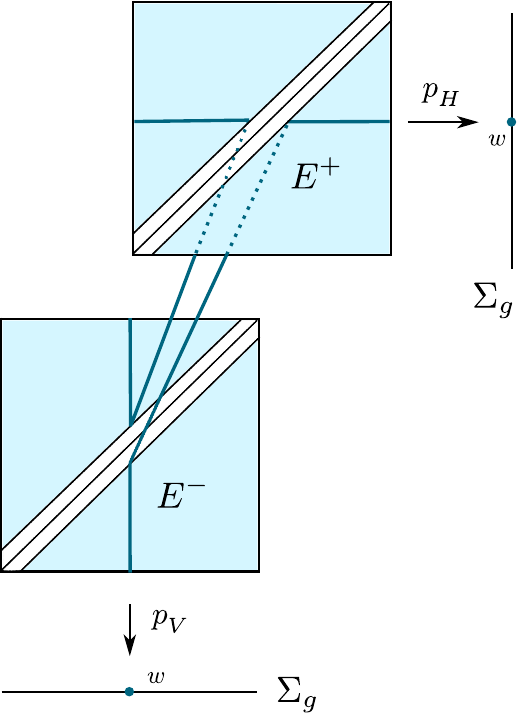}
\caption{A second cartoon sketch of the fibering $p_4$.}
\label{figure:newfibering2}
\end{figure}
\end{center}

\begin{remark}\label{remark:pi1}
As remarked above, the four fiberings constructed above are in fact fiberwise diffeomorphic, by applying factor-swapping involutions $(x,y) \to (y,x)$ on one or more of the components $E^\pm$. This same phenomenon appears for trivial bundles $\Sigma_g \times \Sigma_h$. When $g \ne h$ the projections onto the first and second factors clearly yield inequivalent bundles, as the fibers are not even the same manifold. On the other hand, when $g = h$, the factor-swapping involution yields a bundle isomorphism between the horizontal and vertical projections of $\Sigma_g \times \Sigma_g$. However, in both of these examples the fiberings are not $\pi_1$-fiberwise diffeomorphic. Moreover, Proposition \ref{proposition:fiberings} shows that $\pi_1$-fiberwise diffeomorphism is equivalent to the natural notion of equivalence on the group-theoretic level. For this reason, we believe that $\pi_1$-fiberwise diffeomorphism is an important notion of equivalence for surface bundles over surfaces. By using the techniques of Theorem \ref{theorem:fibervariant}, one can construct surface bundles over surfaces with arbitrarily many fiberings for which the fibers all have distinct genera, and therefore certainly give examples of bundles where the fiberings are not fiberwise diffeomorphic.\end{remark}

\begin{remark}\label{remark:pi1again}
Via the Seifert-van Kampen theorem, it is possible to compute 
\begin{equation} \label{eqn:isom}
\pi_1 E \cong \Gamma *_{\pi_1 UT \Sigma_g} \Gamma,
\end{equation}
where $\Gamma = \pi_1 (\Sigma_g \times \Sigma_g \setminus N)$ and $UT \Sigma_g$ denotes the unit tangent bundle. Let 
\[
\varpi_1: \Sigma_g \times \Sigma_g \setminus N \to \Sigma_g
\]
 denote the vertical projection, and define $\varpi_2$ similarly as the horizontal projection. Relative to the isomorphism of (\ref{eqn:isom}), the induced maps of the four fiberings $(p_i)_*: \pi_1E \to \pi_1\Sigma_g$ correspond to the four amalgamations $(\varpi_i)_* * (\varpi_j)_*: \Gamma *_{\pi_1 UT\Sigma_g} \Gamma \to \pi_1 \Sigma_g$.
\end{remark}
\bigskip

As remarked above, the bundle $p_1: E \to \Sigma_g$ was originally considered by Korkmaz (see Footnote 1 of \cite{margalit}), who constructed its monodromy representation as an example of an embedding $\rho: \pi_1 \Sigma_g \to \mathcal{I}_{2g}$. We now give a description of this embedding. Let $\Mod_g^1$ denote the mapping class group of a surface with one boundary component (where as usual the isotopies are required to fix the boundary component pointwise). We will denote this boundary curve by $\eta$. Consider the embedding
\begin{align*}
f: \pi_1(UT(\Sigma_g)) 		&\to  \Mod_g^1 \times \Mod_g^1\\
\alpha 				&\mapsto (\operatorname{Push}(\alpha), F^{-1}\circ \operatorname{Push}(\alpha)\circ F),
\end{align*}
where $F: \Sigma_g^1 \to \Sigma_g^1$ is any orientation-reversing diffeomorphism. Compose this with the map
\[
h: \Mod_g^1 \times \Mod_g^1 \to \Mod_{2g}
\] 
obtained by juxtaposing the mapping classes $(x,y)$ on the two halves of $\Sigma_{2g}$. Let $\gamma \in \pi_1(UT(\Sigma_g))$ denote the loop around the circle fiber in $UT \Sigma_g$ in the positive direction as specified by the orientation on $\Sigma_g$. The map $\operatorname{Push}(\gamma) \in \Mod(\Sigma_g^1)$ corresponds to a positive twist about $\eta$. We claim that $h(f(\gamma)) = \id$. Indeed, the notion of ``positive'' twist is relative to a choice of orientation, and after the boundary components of the two copies of $\Sigma_g^1$ have been identified, the two twists correspond to a positive and negative twist about $\eta$, and so the result is isotopic to the identity. 

The element $\gamma \in \pi_1(UT(\Sigma_g))$ generates a normal subgroup, and the quotient $\pi_1(UT(\Sigma_g)) / \pair{\gamma} \approx \pi_1 \Sigma_g$. Therefore, we arrive at an embedding $\rho: \pi_1 \Sigma_g \to \Mod_{2g}$ as follows.
\[
\xymatrix{
\pi_1(UT(\Sigma_g)) \ar[r]^-f \ar[d] 	& \Mod_g^1 \times \Mod_g^1 \ar[r]^-h 	& \Mod_{2g}\\
\pi_1\Sigma_g \ar@{.>}[urr]_\rho
}
\]

\begin{lemma}\label{lemma:rhotorelli}
The image of $\rho$ is contained in the Torelli group $\mathcal{I}_{2g}$. 
\end{lemma}
\begin{proof}

 Let $\{\alpha_1, \beta_1, \dots, \alpha_g, \beta_g\}$ be a collection of simple closed curves for which the homology classes $\{[\alpha_1], \dots, [\beta_g]\}$ comprise a generating set for $H_1(\Sigma_g^1)$. Let $F: \Sigma_g^1 \to \Sigma_g^1$ be the orientation-reversing map in the definition of $f$. We can then view $\Sigma_{2g}$ as $\Sigma_g^1 \cup_{\partial \Sigma_g^1} F(\Sigma_g^1)$. Define
\[
\mathcal B = \{\alpha_1, \dots, \beta_g, F(\alpha_1), \dots, F(\beta_g) \}.
\]

It follows that the homology classes $\{[\alpha_1], \dots, [\beta_g], [F(\alpha_1)],\dots, [F(\beta_g)]\}$ comprise a generating set for $H_1(\Sigma_{2g})$. To determine whether a mapping class $\phi \in \Mod(\Sigma_{2g})$ is contained in $\mathcal I_{2g}$, it suffices to show that the homology class of each $\alpha_i, \beta_i, F(\alpha_i), F(\beta_i)$ is preserved by $\phi$. Up to isotopy, $\eta$ is preserved by the action of $\pi_1 \Sigma_g$ via $\rho$, so it suffices to consider how $\pi_1 \Sigma_g$ acts on both copies of $\Sigma_g^1$. If $x \in \pi_1 \Sigma_g$ is given, then on $\Sigma_g^1$, the effect of $\rho(x)$ is to push the boundary component around a loop in $\Sigma_g$ in the homotopy class of $x$. As is well-known (see, for example, \cite{FM}, section 6.5.2), the curves $\eta$ and $\rho(x)(\eta)$ are homologous, for any choice of $x \in \pi_1 \Sigma_g$ and $\eta$ a simple closed curve on $\Sigma_g^1$. In particular, 
\[
[\rho(x)(\alpha_1)] = [\alpha_1], \dots, [\rho(x)(\beta_g)] = [\beta_g],
\]
where these homologies hold in $\Sigma_g^1$ and so necessarily also in $\Sigma_{2g}$. The element $x \in \pi_1 \Sigma_g$ acts on the other half of $\Sigma_{2g}$ via conjugation by $F$, and so similarly the curves $F(\alpha_1), \dots,  F(\beta_g)$ are preserved on the level of homology. As we have shown that each homology class of a generating set for $H_1(\Sigma_{2g})$ is preserved under $\im(\rho)$, it follows that $\im(\rho) \le \mathcal{I}_{2g}$ as claimed. \end{proof}

\begin{theorem}\label{theorem:monodromy}
The monodromy of any of the surface bundle structures $p_i: E \to \Sigma_g\ (i = 1,2,3, 4)$ is the map $\rho: \pi_1 \Sigma_g \to \mathcal{I}_{2g}$ described above. 
\end{theorem}
\begin{proof}
We begin by considering $p_1$. Let $x \in \pi_1\Sigma_g$ be given. The image of the monodromy representation $\mu(x) \in \Mod_{2g}$ is computed by selecting some immersed representative $\gamma$ for $x$, considering the pullback of the bundle $E \to \Sigma_g$ along the immersion map $S^1 \to \Sigma_g$ specified by $\gamma$, and determining the monodromy of this fibered 3-manifold.

The bundle $p_1: E \to \Sigma_g$ is constructed so that the fiber over $w \in \Sigma_g$ consists of two disjoint copies of $\Sigma_{g}$ connect-summed along disks centered at $w$. This means that as one traverses a loop $\gamma \subset \Sigma_g$, the effect of the monodromy is to drag the cylinder connecting the two halves along the loops in either half corresponding to $\gamma$. As a mapping class, this is exactly the map $\rho(x)$ described above.

Now let $\pi_1 E = \Gamma *_{\pi_1 UT \Sigma_g} \Gamma$ as in Remark \ref{remark:pi1again}. There is an involution $\iota: \Gamma \to \Gamma$ induced from the factor-swapping map on $\Sigma_g \times \Sigma_g \setminus \nu(\Delta)$. Let $\varpi_1, \varpi_2$ denote the vertical (resp. horizontal) projection $\Sigma_g \times \Sigma_g \setminus (\nu(\Delta)) \to \Sigma_g$. Then $(\varpi_i)_* \circ \iota = (\varpi_{i+1})_*$ for $i = 1,2$ interpreted mod $2$. As $\iota$ preserves $\pi_1 UT \Sigma_g$, it can be extended to an automorphism of either factor of $\pi_1 E = \Gamma *_{\pi_1 UT \Sigma_g} \Gamma$. In other words, the four surface-by-surface group extension structures on $\pi_1 E$ are in the same orbit of the action of $\Aut(\pi_1 E)$. Consequently, the monodromy representations $r: \pi_1 \Sigma_g \to \Out(\pi_1 \Sigma_{2g})$ are the same. As $r$ is identified with the topological monodromy representation $\rho: \pi_1 \Sigma_g \to \Mod_{2g}$ under the Dehn-Nielsen-Baer isomorphism $\Mod_{2g} \approx \Out^+(\pi_1 \Sigma_{2g})$, this shows that any of the four monodromy representations are equal. 
\end{proof}

We summarize the results of the basic construction in the following theorem.
\begin{theorem}\label{theorem:threefiberings}
For any $g \ge 2$, there exists a $4$-manifold $E$ which admits four fiberings $p_i: E \to \Sigma_g, i = 1,2,3, 4$ as a $\Sigma_{2g}$-bundle over $\Sigma_g$ that are pairwise distinct up to $\pi_1$-fiberwise diffeomorphism. For each $i$, the monodromy $\rho_i: \pi_1 \Sigma_g \to \Mod_{2g}$ of $p_i: E \to \Sigma_g$ is contained in the Torelli group $\mathcal{I}_{2g}$. 
\end{theorem}

\bigskip

\para{Surface bundles over surfaces with $n$ distinct fiberings} We next extend the construction given in the previous subsection to yield examples of surface bundles over surfaces with $n$ distinct (up to $\pi_1$-fiberwise diffeomorphism) fiberings for arbitrary $n$. Let $X$ be a connected bipartite graph with vertex set $V(X)$ and edge set $E(X)$ of cardinalities $C, D$ respectively. As $X$ is bipartite, it admits a coloring $c: V(X) \to \{+,-\}$ in such a way that if $v$ is colored with $\pm$, then all the vertices $w$ adjacent to $v$ are colored $\mp$. Consequently we define $\delta^\pm: E(X) \to V(X)$ be the map which sends $e$ to the vertex $v \in e$ colored $\pm$. 

Let $G$ be a finite group with $\abs{G} = n$, where $n$ is an integer such that every $v \in V(X)$ has valence at most $n$. Assign labelings $g^\pm: E(X) \to G$ to the half-edges of $X$, subject to the restriction that $g^\pm$ is an injection when restricted to 
\[
\{e \in E(X) \mid \delta^\pm(e) = v\}
\]
for any $v \in V(X)$. In other words, the set of half-edges adjacent to any vertex must have distinct labelings. See Figure \ref{figure:graph}.

\begin{center}
\begin{figure}
\includegraphics{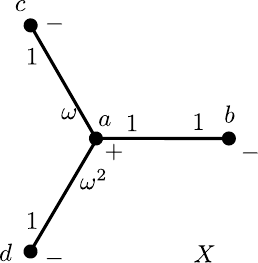}
\caption{An example of a graph $X$ equipped with a labeling of the half-edges by elements of $G = \Z / 3 \approx \{1, \omega, \omega^2\}$ the group of third roots of unity.}
\label{figure:graph}
\end{figure}
\end{center}

Let $\Sigma$ be a surface admitting a free action of $G$, such as the one depicted in Figure \ref{figure:surface}. For each $v \in V(X)$, consider the $4$-manifold $E_1^v = \Sigma\times \Sigma$, oriented so that the orientations on $E_1^v$ and $E_1^w$ disagree whenever $c(v) \ne c(w)$. Each $E_1^v$ admits two projections $p^{v,1}, p^{v,2}: E_1^v \to \Sigma_g$ onto the first (resp. second) factor.

For $x \in G$, let 
\[
\Delta^x = \{(w, x\cdot w) \mid w \in \Sigma\} \subset \Sigma \times \Sigma
\]
be the graph of $x: \Sigma \to \Sigma$. By abuse of notation we can view $\Delta^x$ as embedded in any of the $E_1^v$. Let $\Delta$ be the disconnected surface embedded in $E_1 = \bigcup_{v \in V(X)} E_1^v$ for which 
\[
\Delta \cap E_1^v = \bigcup_{v \in e} \Delta^{g^{c(v)}(e)}.
\]
Let $N$ denote the $\epsilon$-neighborhood of $\Delta$. There is a decomposition
\[
N = \bigcup_{e \in E(X)} N^e
\]
and a further decomposition
\[
N^e = N^{e, +} \cup N^{e, -} \qquad \mbox{ with }\qquad N^{e, \pm} \subset E_1^{\delta^\pm(e)}.
\]
Each $N^{e, \pm}$ is the $\epsilon$-neighborhood of a single component of $\Delta$. 

\begin{center}
\begin{figure}
\includegraphics{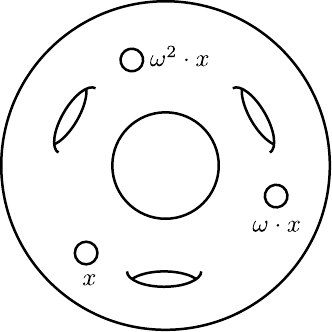}
\caption{A surface $\Sigma$ admitting a free action of $G = \{1, \omega, \omega^2\}$. With respect to the labeling in Figure \ref{figure:graph}, the fiber of $E_2^a$ over $x \in \Sigma$ has neighborhoods of $x, \omega\cdot x,$ and $\omega^2 \cdot x$ removed. }
\label{figure:surface}
\end{figure}
\end{center}

Define
\[
E_2 = E_1 \setminus \operatorname{int}(N)
\]
and, for $v \in V(X)$,
\[
E_2^v = E_2 \cap E_1^v.
\]
The orientation convention ensures that for each $e \in E$, the Euler numbers of the disk bundles $N^{e, \pm}$ are given by $\pm \chi(\Sigma)$. Their boundaries can therefore be identified via an orientation-reversing diffeomorphism. As in the previous construction, it  will be convenient to specify the gluing maps only up to isotopy, and as before we will take the isotopy class of the identity. 

With these conventions in place, we define the (connected oriented) $4$-manifold
\[
E_X = \bigcup_{v \in V(X)} E_2^v
\]
glued together as prescribed by the labeled graph $X$ with all identifications of boundary components in the isotopy class of the identity.
Figure \ref{figure:surface} depicts a portion of the fiber of $E_X$ for the graph $X$ of Figure \ref{figure:graph}. Figure \ref{figure:totspc} depicts the total space of $E_X$. The portion of the fiber shown in Figure \ref{figure:surface} is the portion contained in the central component of Figure \ref{figure:totspc}.

\begin{center}
\begin{figure}
\includegraphics{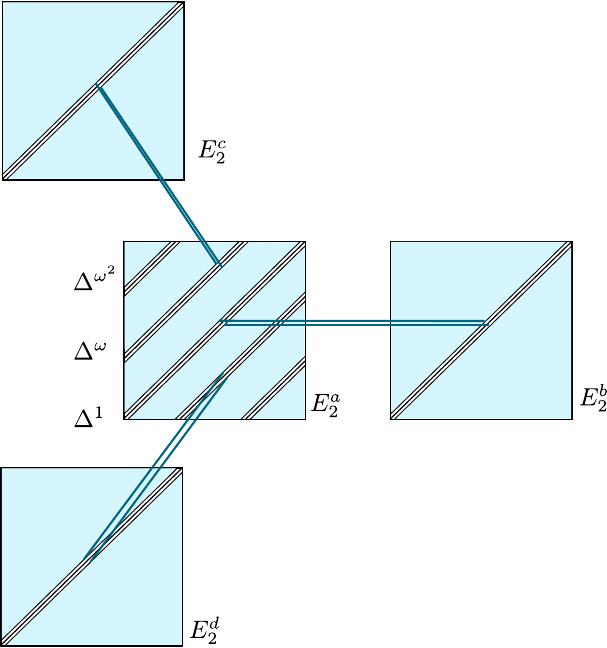}
\caption{A schematic rendering of the $4$-manifold $E_X$ associated to the graph $X$ of Figure \ref{figure:graph} and the surface $\Sigma$ of Figure \ref{figure:surface}. The lines connecting the components indicate how the various $\tilde N^{e}$ are attached.}
\label{figure:totspc}
\end{figure}
\end{center}

\begin{theorem}\label{theorem:nfiberings}
Let $X$ be a finite bipartite graph, possibly with multiple edges, with vertex set $V(X)$ and edge set $E(X)$ of cardinalities $C,D$ respectively. Then,
\begin{packed_enum}
\item The manifold $E_X$ constructed above admits $2^C$ fiberings $p^f: E \to \Sigma$ as a surface bundle over a surface, indexed by the set of maps $f: V(X) \to \{1,2\}$. The fiberings are pairwise-inequivalent up to $\pi_1$-fiberwise diffeomorphism. 
\item  \label{item:fiber}The fiber of any of the fiberings is a surface of the form $\tilde \Sigma = \Sigma^{\# C} \# \Sigma_{1 - C + D}$.
\item The total space $E_X$ has the structure of a graph of groups modeled on $X$ where the vertex groups are free-by-surface group extensions $\Gamma$ and the edge groups are given by $\pi_1 UT \Sigma$ (with notation as in Remark \ref{remark:pi1again}).
\end{packed_enum}
 \end{theorem}

\begin{proof}
Given $f: V(X) \to \{1,2\}$, define $p^f$ on each component $E_2^v$ via $p^{v,f(v)}$. To realize $p^f$ as a smooth map, it is necessary to specify gluing maps identifying the various components of $E_2$, as well as appropriate collar neighborhoods. We proceed exactly as in Theorem \ref{theorem:threefiberings}. For each $x \in G$, there is an identification of (neighborhoods of) $\Delta^x$ with $\Delta^1$ via the action of the diffeomorphism $\id \times x^{-1}$ of $\Sigma \times \Sigma$. Relative to these identifications, we will speak of identifying $\partial(N^{e, +})$ and $\partial(N^{e, -})$ via $\id$ or by $h_1$ as in Theorem \ref{theorem:threefiberings}. Likewise, we will speak of the collar neighborhoods $\theta_1$ and $\theta_2$ of $\partial(N^{e, \pm})$ (referred to as $\theta_V$ and $\theta_H$ respectively in Theorem \ref{theorem:threefiberings}). 

The identifications are indexed via $E(X)$. As in Theorem \ref{theorem:threefiberings}, identify $\partial(N^{e, +})$ and $\partial(N^{e, -})$ via $\id$ if $f(\delta^+(e)) = f(\delta^-(e))$ and via $h_1$ otherwise. Then a collar neighborhood of $\partial(N^{e, \pm})$ for which $p^f$ is smooth is given by $\theta_{f(\delta^{\pm}(e))}$. 

The argument that each of the fiberings are distinct up to $\pi_1$-fiberwise diffeomorphism proceeds along the same lines as in Theorem \ref{theorem:distinct}. If $f_1, f_2: V(X) \to \{1,2\}$ are distinct, then there exists at least one $v$ for which $f_1(v) \ne f_2(v)$. Arguing as in Theorem \ref{theorem:distinct}, one produces an essential loop $\gamma \subset E_2^v$ contained in the fiber of $f_1$ that projects onto an essential loop under $f_2$. 

By definition, a graph of groups on a graph $X$ is constructed by connecting Eilenberg-Mac Lane spaces $K(\Gamma_v, 1)$ indexed by the vertices, along mapping cylinders induced from homomorphisms $\phi_e: \Gamma_e \to \Gamma_v$. In our setting, for each $v \in V(X)$, the space $E_2^v$ is a $K(\pi_1 E_2^v,1)$ space, since it is the total space of a fibration $\Sigma' \to E_2^v \to \Sigma$, where $\Sigma'$ is obtained from $\Sigma$ by removing $n$ open disks, one for each edge incident to $v$. As the base and the fiber of this fibration are both aspherical, it follows from the homotopy long exact sequence that $E_2^v$ is aspherical as well. The edge spaces are given by $\partial(N^{e, \pm})$, each of which is diffeomorphic to the aspherical space $UT \Sigma$. It follows that $E_X$ is indeed a graph of groups. \end{proof}

\begin{remark}
In contrast with the construction in Theorem \ref{theorem:threefiberings}, the monodromy representations associated to an arbitrary $E_X$ need not be contained in the Torelli group. For example, let $X$ be a graph with two vertices and two edges connecting them. We can take $\Sigma$ to be a surface of genus $3$. Then it is easy to find elements of the monodromy that do not preserve the homology of the fiber. See Figure \ref{figure:monodromy}.

It can also be seen from this point of view that the images of the monodromy representations will be contained in the {\em Lagrangian mapping class group} $\mathcal{L}_g$, defined as follows. The algebraic intersection pairing endows $H_1(\Sigma_g, \Z)$ with a symplectic structure, and there is a decomposition
\[
H_1(\Sigma_g, \Z) = L_x \oplus L_y
\]
as a direct sum, with the property that the algebraic intersection pairing restricts trivially to $L_x$ and to $L_y$. Then 
\[
\mathcal{L}_g := \{f \in \Mod_g \mid f(L_x) = L_x\}.
\]

Suppose $\tilde \Sigma$ has been constructed from a finite graph $X$ as in Theorem \ref{theorem:nfiberings}. Let $\rho:\pi_1 \Sigma \to \Mod(\tilde \Sigma)$ be the associated monodromy. There is a Lagrangian subspace of $H_1(\tilde \Sigma)$ of the form
\[
L = \bigoplus_{v \in V(X)}L_v \oplus \mathcal C,
\]
where $L_v$ is a Lagrangian subspace of the fiber of $E_2^v$, and $\mathcal C \le H_1(\tilde \Sigma)$ is the (possibly empty) subspace generated by the homology classes of the former boundary components in $\tilde \Sigma$. By construction, for all $x \in L_v$ and all $g \in \pi_1(\Sigma)$, the equation
\[
\rho(g)(x)= x + c
\] 
holds in $H_1(\tilde \Sigma)$, for some appropriate $c \in \mathcal C$. As $\mathcal C$ is fixed elementwise by the action of $\rho$, it follows that $L$ is indeed a $\rho$-invariant Lagrangian subspace.

In \cite{sakasai}, Sakasai showed that the first MMM class $e_1 \in H^2(\Mod_g, \Z)$ vanishes when restricted to $\mathcal{L}_g$. It follows that the surface bundles over surfaces constructed in this section all have signature zero. More generally, suppose $\Sigma_g \to E \to \Sigma_h$ is a surface bundle over a surface with monodromy representation $\rho: \pi_1 \Sigma_h \to \Gamma$, where $\Gamma \le \Mod_g$ is a subgroup. We can view the bundle $E \to \Sigma_h$ as giving rise to a homology class $[E] \in H_2(\Gamma, \Z)$, e.g. by taking the pushforward $\rho_*([\Sigma_h])$ of the fundamental class. 
\end{remark}
\begin{center}
\begin{figure}
\includegraphics[scale=0.8]{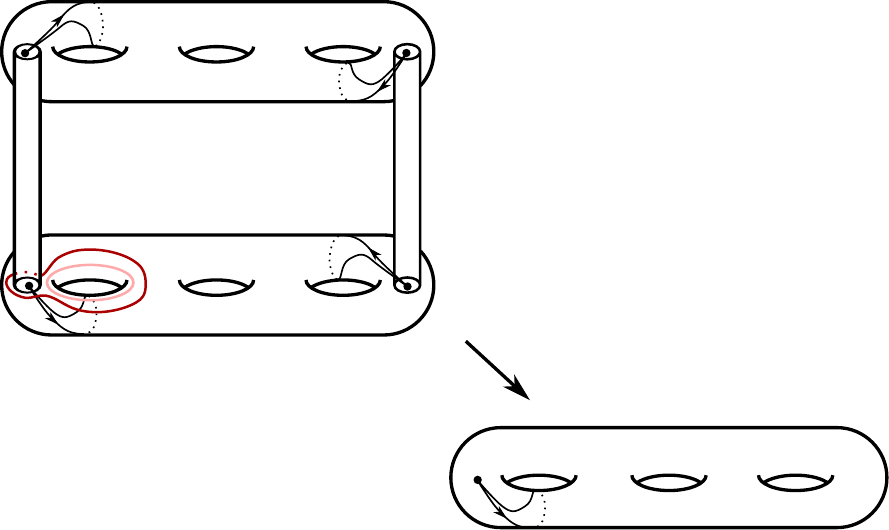}
\caption{The lighter curve is taken to the darker one under the monodromy action associated to the loop on the base surface. The dark and the light curves are not homologous. The identifications of the boundary components have been indicated by cylinders.}
\label{figure:monodromy}
\end{figure}
\end{center}
\begin{question}
Do the examples of surface bundles over surfaces given in Theorem \ref{theorem:nfiberings} determine nonzero classes in $\mathcal{L}_g$? For a fixed $g$, what is the dimension of the space spanned in $H_2(\mathcal{L}_g, \Q)$ by the examples in Theorem \ref{theorem:nfiberings} with fiber genus $g$?
\end{question}

\para{Further constructions} It is possible to extend the constructions in Theorem \ref{theorem:threefiberings} and Theorem \ref{theorem:nfiberings} to obtain examples where the base and fibers of distinct fiberings do not all have the same genus. The author is grateful to D. Margalit for suggesting the basic idea underlying the constructions in this subsection. 

 \begin{theorem}\label{theorem:fibervariant}
 Let $\Sigma$ be a surface admitting a free action by a finite group $G$ of order $n$, let $X$ be a connected bipartite graph of maximal valence $n$, and let $f^v: \tilde \Sigma \to \Sigma^v$ for $v \in V(X)$ be covering maps, not necessarily distinct. Then there exists a $4$-manifold $E_X$ admitting $\abs{V(X)} + 1$ fiberings $p^0, p^v (v \in V(X))$, with $p^0: E_X \to \Sigma$ and $p^v: E_X \to \Sigma^v$ all projection maps for surface bundle structures on $E$, distinct up to $\pi_1$-fiberwise diffeomorphism. If the surfaces $\Sigma^v$ and $\Sigma^w$ have distinct genera, the fiberings $p^v, p^w$ are distinct up to fiberwise diffeomorphism.
 \end{theorem}

\begin{proof} Let $\Sigma^0$ be a closed surface of genus $g$ that admits coverings $f^1: \Sigma^0 \to \Sigma^1$ and $f^2: \Sigma^0 \to \Sigma^2$ of degree $d_1, d_2$ respectively. For $i = 1,2$, consider the graphs $\Gamma_i \subset \Sigma^0 \times \Sigma^i$ of the coverings $f^i$. Thicken these to tubular neighborhoods $N^i$. Each $\partial N^i$ is an $S^1$-bundle over $\Sigma^0$ with Euler number $\chi(\Sigma^0)$. By reversing the orientation on one of the components, it is therefore possible to fiberwise connect-sum $\Sigma^0 \times \Sigma^1$ and $\Sigma^0 \times \Sigma^2$ along $N^1$ and $N^2$ to make the $4$-manifold $E$.

Let $p_V: E_2 \to \Sigma^0$ and $p_H^i: E_2^i \to \Sigma^i$ be the vertical and horizontal projections. These can be combined in various ways to define three distinct fiberings on $E$. The first fibering $p_0: E \to \Sigma^0$ is  given by the projection onto the first factor on both coordinates of $E_2$, so that the fiber is $\Sigma^1 \# \Sigma^2$. The second fibering $p_1: E \to \Sigma^1$ is given by $p_H^1$ on $E_2^1$, and by $f^1 \circ p_V$ on $E_2^2$. Let $F_1$ denote the fiber of $p_1$ over $w \in \Sigma^1$. Then (relative to an appropriate metric $d$ and a suitable $\epsilon > 0$)
 \[
 F_1 \cap E_2^1 = \{(y,w) \in \Sigma^0 \times \Sigma^1 \mid d(f^1(y), w) \ge \epsilon\}
 \]
 is a copy of $\Sigma^0$ with $d_1$ disks removed (recall that $d_i$ is the degree of the covering $f^i: \Sigma^0 \to \Sigma^i$). In turn,
 \[
 F_1 \cap E_2^2 = \{(v, y) \in \Sigma^0 \times \Sigma^2 \mid f^1(v) = w, d(f^2(v),y) \ge \epsilon\}
 \]
 consists of $d_1$ copies of $\Sigma^2$, each with one boundary component. In total then,
 \[
 F_1 = \Sigma^0 \#\left(\Sigma^2 \right)^{\# d_1}.
 \]
 
 When $d_1 > 1$, the monodromy of $p_1$ is not contained in the Torelli group $\mathcal I_{g}$. Let $\gamma$ be a loop on $\Sigma^1$ which lifts to an arc $\tilde \gamma \subset \Sigma^0$ with endpoints $v_1, v_2$. Then the component of $F_1 \cap E_2^2$ lying over $v_1 \in \Sigma^0$ is sent to the component lying over $v_2$. If $x$ is a loop in the first component representing some nontrivial homology class in $F_1$, then $\rho(\gamma)(x)$ is a distinct homology class in $F_1$, and so the monodromy of $p_1$ has a nontrivial action on $H_1(\Sigma_g, \Z)$.

  The construction of $p_2: E \to \Sigma^2$ is completely analogous. The fibering $p_2$ is given by $f^2 \circ p_V$ on $E_2^1$ and by $p_H^2$ on $E_2^2$. The fiber is of the form
 \[
 F_2 = \Sigma^0 \# \left(\Sigma^1\right)^{\# d_2}.
 \]
 
 As in the previous constructions it is necessary to specify the precise identification maps as well as collar neighborhoods. The internal details proceed along similar lines as before, except that the boundary identifications require some further comment. Realize $\partial N^i$ as a subset of $\Sigma^0 \times \Sigma^i$. Then $\partial N^i$ is the total space of two different fiber bundle structures inherited respectively from $p_V$ and $p_H^i$. The identification maps for the various $p_i$ will be constructed so as to preserve fibers of these various fiberings. 
 
 For $p_0$, identify $\partial N^1$ and $\partial N^2$ in a fiber-preserving way with respect to $p_V$ on both $\partial N_1$ and $\partial N_2$. For $p_1$, identify $\partial N^1$ and $\partial N^2$ so that $p_H^1$-fibers on $\partial N_1$ correspond to $p_V$-fibers $\partial N_2$. More precisely, given $z \in \Sigma^1$, the $p_H^1$-fiber of $z$ consists of $d_1$ disjoint circles projecting down to circles in $\Sigma^0$ centered at the points of $(f^1)^{-1}(z)$. For every $x \in \Sigma^0$, the identification of $\partial N_1$ and $\partial N_2$ identifies $p_V^{-1}(x)$ with the component of $(p_H^1)^{-1}(f^1(x))$ centered over $x$. The identification of $\partial N^1, \partial N^2$ appropriate for $p_2$ is constructed analogously, matching $p_V$-fibers of $\partial N_1$ with $p_H^2$-fibers of $\partial N_2$. 
 
 The straight-line isotopy $h_t$ constructed in the course of Theorem \ref{theorem:threefiberings} was purely local in its definition. The same formulas as before show that the three gluing maps constructed in the above paragraph are mutually isotopic, and the construction proceeds as before.\\

 It is also possible to generalize the construction of Theorem \ref{theorem:nfiberings}, so that the surfaces used in the construction of $E_X$ are all covered by $\Sigma$. For $v \in V(X)$, let $f^v: \Sigma \to \Sigma_v$ be a covering. Suppose that each $\Sigma_v$ admits a free action of a group $G_v$, such that $\abs{G_v}$ is at least the valence of $v$. We may then repeat the construction of Theorem \ref{theorem:nfiberings}, taking $E_1^v = \Sigma \times \Sigma_v$. Since $G_v$ acts freely, for $g,h \in G_v$, the graphs of $g \circ f^v$ and $h \circ f^v$ are disjoint as submanifolds of $E_1^v$. We may then remove neighborhoods of these graphs to produce $E_2^v$ and connect the boundaries as in Theorem \ref{theorem:nfiberings}. The resulting $E_X$ has at least $\abs{V(X)} + 1$ fiberings $p_0, p_v (v \in V(X))$.  The first fibering $p_0$ is defined on each $E_2^v$ via $p_V$, and the result is a fiber bundle $p_0: E_X \to \Sigma$. For $v \in V(X)$, define $p_v$ on the components $E_2^v$ via
 \[
 \left. p_v \right |_{E_2^w} \begin{cases}
 					p_H^v		& w = v\\
					f^v \circ p_V	& w \ne v.
 \end{cases}
 \]
 The result is a fibering $p_v: E_X \to \Sigma^v$. 
\end{proof}
 
\begin{example}
Let $\Sigma$ be a surface admitting a free action of $\Z / 2^{n}$ for some $n$. For $0 \le k \le n$ define $\Sigma^k = \Sigma / (\Z/2^k)$. Let $f^k: \Sigma \to \Sigma^k$ be the associated covering. Each $\Sigma^k$ admits an action of $\Z/2^{n-k}$, so that for $k \le n-1$, each $\Sigma^k$ admits a free involution $\tau^k$. Let $X$ be the ``line graph'' with vertex set $V(X) = \{0,1, \dots, n\}$, such that $\{i,j\} \in E(X)$ whenever $\abs{i-j} = 1$.

In this setting, the construction of Theorem \ref{theorem:fibervariant} produces a $4$-manifold $E^4$ which fibers as a surface bundle over $\Sigma^k$ for each $0 \le k \le n$. In more detail, define $E_1^k = \Sigma \times \Sigma^k$. For $0 \le k \le n-1$, the graphs of $f^k$ and $\tau^k \circ f^k$ are disjoint, and we attach $E_1^k$ to $E_1^{k+1}$ by joining the graph of $\tau^k \circ f^k \subset E_1^k$ to the graph of $f^{k+1} \subset E_1^{k+1}$. Although $E_1^n$ does not necessarily admit a free involution, the vertex $n \in X$ has valence $1$, and $E_1^{n-1}$ can still be joined to $E_1^n$ using the rule described above, resulting in a $4$-manifold $E_X$.

For $0 \le k \le n$, there are fiberings $p_k: E_X \to \Sigma^k$ defined on components $E_2^j \subset E_X$ via
\[
\left. p_k \right |_{E_2^j} = \begin{cases} 	p_H^k 		& j = k\\
								f^k \circ p_V	& j \ne k
\end{cases}
\]
Together, these realize $E_X$ as the total space of a surface bundle over $\Sigma^k$ for each $0 \le k \le n$.

\end{example}

\section{Further questions}\label{section:questions}
In this final section we collect together some questions about surface bundles over surfaces with multiple fiberings. Our first line of inquiry concerns the number of possible fiberings that surface bundles over a surface with given Euler characteristic can admit. 

\begin{proposition}\label{proposition:bound}
Let $E^4$ be a $4$-manifold with $\chi(E) = 4d$. Then $E$ admits at most\footnote{In fact, an additional argument, such as the one given in section 5.2 of \cite{hillman}, can be used to obtain the slightly better bound $\sigma_0(d) d^{2d+6}$. The bound given here is good enough for our purposes.}
\[
F(d) = \sigma_0(d) (d+1)^{2d+6}
\]
fiberings as a surface bundle over a surface which are distinct up to $\pi_1$-fiberwise diffeomorphism, where $\sigma_0(d)$ denotes the number of divisors of $d$. 
\end{proposition}

\begin{proof}
To obtain the explicit bound given above, we will first reproduce F.E.A. Johnson's original argument, incorporating some improvements suggested by J. Hillman. Let $p: E \to \Sigma_h$ be the projection for a $\Sigma_g$-bundle structure on $E$. There is an associated short exact sequence of fundamental groups
\begin{equation}\label{eqn:ses}
1 \to K \to \pi_1 E \to \pi_1 \Sigma_h \to 1,
\end{equation}
with $K \approx \pi_1 \Sigma_g$ the fundamental group of the fiber.

We will first show that if $g < h$, then $p$ determines the unique $\Sigma_g$-bundle structure on $E$, up to $\pi_1$-fiberwise diffeomorphism. Equivalently (by Proposition \ref{proposition:fiberings}), it suffices to show that (\ref{eqn:ses}) is the unique splitting of $\pi_1 E$ as an extension of $\pi_1 \Sigma_h$ by $\pi_1 \Sigma_g$. 

Suppose $p': E \to \Sigma_h$ is a second fibering, giving rise to a short exact sequence
\[
1 \to K' \to \pi_1 E \to \Sigma_h \to 1.
\]
Consider the projection $p_*|_{K'}$. Suppose first that $p_*(K') = \{1\}$, or equivalently $K' \le \ker p_* = K$. As $K$ and $K'$ are both isomorphic to $\pi_1 \Sigma_g$, in this case $K = K'$.

Suppose next that $\im(p_*|_{K'})$ is nontrivial. In this case, the image $p_*(K')$ is a nontrivial finitely generated normal subgroup of the surface group $\pi_1 \Sigma_h$. It is a general fact that if $N \normal \pi_1 \Sigma_h$ is any nontrivial finitely-generated normal subgroup, then $N$ has finite index in $\pi_1 \Sigma_h$ (cf Theorem 3.1 of \cite{rivinfiber}). No finite-index subgroup of $\pi_1 \Sigma_h$ is generated by strictly fewer than $2h$ generators. On the other hand, $K'$ is generated by $2g$ generators by assumption. This is a contradiction, and it follows that $\im(p_*|_{K'}) = \{1\}$. By the argument of the previous paragraph, this shows that necessarily $K = K'$, and so $p: E \to \Sigma_h$ is the unique $\Sigma_g$-bundle structure on $E$ as claimed.

Returning to the general setting, suppose $p: E \to \Sigma_h$ is a $\Sigma_g$-bundle over $\Sigma_h$. As before, let $K \approx \pi_1 \Sigma_g$ denote the fundamental group of the fiber. The Euler characteristic is multiplicative for fiber bundles:
\[
\chi(E) = \chi(\Sigma_g) \chi(\Sigma_h) = 4(g-1)(h-1).
\] 
Let $d = (g-1)(h-1)$, so that $\chi(E) = 4d$. Any $d+1$-sheeted cover of $\Sigma_h$ has genus $(h-1)d + h = (h-1)^2(g-1)+h$, and this quantity is strictly larger than $g$. Let $\tilde \Sigma \to \Sigma_h$ be such a cover, and let $\tilde p: \tilde E \to \tilde \Sigma$ denote the pullback of $p$ along this cover. Then $\tilde p$ has the property that the genus of the fiber is strictly smaller than the genus of the base. By the above argument, $K$ is the unique normal subgroup of $\pi_1 \tilde E$ isomorphic to $\pi_1 \Sigma_g$ with surface group quotient. 

Let $\tilde \alpha: \pi_1 E \to \Z / (d+1) \Z$ be an epimorphism. If $\tilde \alpha(K) = 0$, then $\tilde \alpha$ is induced from a map $\alpha: \pi_1 \Sigma_h \to \Z / (d+1) \Z$. Let $\tilde \Sigma$ denote the cover of $\Sigma_h$ associated to $\alpha$. Carrying out the construction of the previous paragraph, it follows that to each such $\tilde \alpha$ there is at most one $\Sigma_g$-bundle structure on $E$. As $\chi(\Sigma_g)$ must divide $\chi(E)$, it follows that $E$ can be the total space of a $\Sigma_g$-bundle for only finitely many $g$. As $\Hom(\pi_1 E, \Z / (d+1) \Z)$ is finite, this completes the portion of the argument due to F.E.A. Johnson.

Our own extremely modest contribution to Proposition \ref{proposition:bound} is to determine an explicit upper bound on the maximal cardinality of $\Hom(\pi_1 E, \Z / (d+1) \Z)$ over all possible surface bundles $E$ of a fixed Euler characteristic $4d$. It follows from (\ref{eqn:ses}) that a surface bundle $\Sigma_g \to E \to \Sigma_h$ admits a generating set for $\pi_1 E$ of size $2g + 2h$. As $g,h$ range over all possible pairs such that $(g-1)(h-1) = d$, the largest value of $2g + 2h$ is obtained for $g = d+1, h = 2$. This shows that any surface bundle over a surface $E$ with $\chi(E) = 4d$ has a generating set with at most $2d + 6$ generators. It follows that 
\[
\abs{\Hom(\pi_1 E, \Z /(d+1)\Z)} \le (d+1)^{2d+6}.
\]
As noted above, for each $\alpha \in \Hom(\pi_1 E, \Z /(d+1)\Z)$, the corresponding cover $\tilde E$ has at most one $\Sigma_g$-bundle structure for each $g \ge 2$ such that $g -1$ divides $d$. The bound in the statement of the Proposition follows. \end{proof}

We defined the function $N(d)$ in the Introduction,
\[
N(d) := \max\left\{\begin{array}{c|c}
n 	&	\mbox{there exists $E^4, \chi(E) \le 4d$, $E$ admits $n$ surface bundle structures}\\
 	& \mbox{distinct up to $\pi_1$-fiberwise diffeomorphism.}\end{array}
\right\}
\]
Proposition \ref{proposition:bound} shows that $N(d) \le \sigma_0(d) (d+1)^{2d+6}$. Prior to the results of this paper, the best known lower bound on $N(d)$ was $N(d) \ge 2$. Drastic improvements can be made by making use of the construction of Theorem \ref{theorem:nfiberings}. Let $\Sigma$ be a surface of genus $3$ admitting a free involution $\tau$, and let $X$ be the ``line graph'' with vertex set $V(X) = \{1, 2, \dots, n\}$, such that $\{i,j\} \in E(X)$ whenever $\abs{i-j} = 1$. According to Theorem \ref{theorem:nfiberings}, the corresponding $E_X$ has $2^n$ fiberings. For each fibering, the base has genus $3$ and the fiber has genus $3n$; consequently $\chi(E_X) = 4\cdot 2 \cdot (3n-1)$. This shows that 
 \[
 N(6n-2) \ge 2^n.
 \]
 Combining this with Johnson's upper bound, we obtain
 \[
 2^{(d+2)/6} \le N(d) \le \sigma_0(d) (d+1)^{2d+6}.
 \]

\begin{problem}
Study the function $N(d)$. Sharpen the known upper bounds on $N$, and construct new examples of surface bundles over surfaces to improve the lower bounds. 
\end{problem}

One feature of the constructions given here is that they all take place within the smooth category, and cannot be given complex or algebraic structures. Indeed, all of the monodromy representations of the constructions of Section \ref{section:construction} globally fix the isotopy class of a curve contained in the fiber (one of the former boundary components). H. Shiga has shown (\cite{shiga}) that if $E$ is a $4$-manifold with a complex structure, $B$ a Riemann surface, and $p: E \to B$ a holomorphic map realizing $E$ as the total space of a holomorphic family of Riemann surfaces, then the monodromy cannot globally fix the isotopy class of any curve. On the other hand, it has been shown independently by J. Hillman, M. Kapovich, and D. Kotschick (cf. Theorem 13.7 of \cite{hillman}) that if $E$ and $B$ are as above and $p: E \to B$ is a {\em smooth} fibration of $E$ over $B$, then there exists a holomorphic map $p': E \to B$ that realizes $E$ as the total space of a holomorphic family of Riemann surfaces. Combining these results with the known reducibility of the monodromies of the examples in this paper, one sees that our examples cannot be given complex structures. On the other hand, the examples of Atiyah and Kodaira that admit two fiberings take place in the algebraic category, prompting the following.

\begin{question}
Let $E^4$ be a complex surface that is the total space of a surface bundle over a surface $p: E \to X$. Can such an $E$ admit three or more such fiberings? More generally, can a $4$-manifold with nonzero signature admit three or more structures as a surface bundle over a surface?
\end{question}

This question is closely related to a point raised briefly in the introduction, and we remark that it is possible that the list of known fiberings of a given $4$-manifold need not be exhaustive. There can be ``hidden'' fiberings that are not immediately apparent.
\begin{question}\label{question:known}
Are the two known fiberings of surface bundles over surfaces of the Atiyah-Kodaira type the only surface bundle structures on these manifolds? Do the manifolds constructed in Section \ref{section:construction} admit more fiberings than described in this paper? Is there some finite-sheeted cover of an Atiyah-Kodaira manifold that admits three or more fiberings?
\end{question}

    	\bibliography{multifiber}{}
	\bibliographystyle{alpha}

\end{document}